\documentclass[a4paper,12pt]{article}
\usepackage[utf8]{inputenc}
\usepackage[ english]{babel}
\usepackage{textcomp} 
\usepackage{amsmath, amsthm, amsfonts, amssymb}
\usepackage{marvosym}
\newcommand{\R}{{\mathbb{R}}}

\newcommand{\cov}{\mathop{\rm cov}}

\newcommand{\ov}{\overline}

\newcommand{\vf}{\varphi}

\newcommand{\cE}{{\mathcal E}}

\newcommand{\mbR}{{\mathbb R}}

\theoremstyle{plain}
\newtheorem{thm}{Theorem}
\newtheorem{lem}[thm]{Lemma}
\newtheorem{example}{Example}[section]
\newtheorem{remark}{Remark}[section]

\theoremstyle{definition}
\newtheorem{defn}[thm]{Definition}
\newtheorem{prop}[thm]{Proposition}
\begin{document}

\begin{center}
\textbf{Clark formula for local time for one class of Gaussian processes}
\end{center}

\paragraph*{Abstract \\ \\}

In the article we present chaotic decomposition and analog of the Clark formula for the local time of Gaussian integrators. Since  the integral with respect to Gaussian integrator is understood in Skorokhod sense, then there exist more than one Clark representation for the local time. We present different representations and discuss the representation with the minimal $L^2$-norm.
\begin{flushleft}

\end{flushleft}
\paragraph*{Authors  }

\begin{flushleft}
   A. A. Dorogovtsev, Institute of Mathematics, National Academy of Sciences of Ukraine, Ukraine \\
\textit{E-mail address:  andrey.dorogovtsev\MVAt gmail.com }
\end{flushleft}

\begin{flushleft}
 O. L. Izyumtseva, Institute of Mathematics, National Academy of Sciences of Ukraine, Ukraine\\
\textit{E-mail address:  olaizyumtseva\MVAt yahoo.com }
\end{flushleft}

\begin{flushleft}
 Georgii Riabov, Institute of Mathematics, National Academy of Sciences of Ukraine, Ukraine\\
\textit{E-mail address:  ryabov.george\MVAt gmail.com } 
\end{flushleft}

\begin{flushleft}
Naoufel Salhi, Faculty of Sciences of Tunis, University of Tunis El Manar, Tunisia \\
\textit{E-mail address:  salhi.naoufel\MVAt gmail.com }
\end{flushleft}

  \paragraph*{Introduction \\ \\}
In this paper we establish some integral representations for the local time of Gaussian integrators. Recall that for the one dimensional Wiener process $ \lbrace w(t), \; 0 \leqslant t \leqslant 1 \rbrace ,$ the local time at a point $u$ up to time $t$ can be formally denoted by $ \displaystyle \int_0^t \delta_u (w(s))ds $ and possess a meaning using approximations of the Dirac delta function. It was proved that the following limit exists  
$$ \int_0^t \delta_u (w(s))ds := L^2 - \lim _{\varepsilon \rightarrow 0^{+}} \int_{0}^{t}p_{\varepsilon}(w(s)-u)ds $$
where $ p_\varepsilon (x)=  (2\pi \varepsilon)^{-1/2}\exp(-x^2/2\varepsilon)  $ \cite{15} .
  
  Another approach is based on the conception of the local time as the density of the occupation measure 
$$ \mu_t (D) = \int_0^t \mathbf{1}_D(w(s))ds \; , \quad D \; \in \; \mathcal{B}(\mathbb{R}) . $$
L\'{e}vy proved in \cite{13}, that, almost surely, $ \mu_t $ is absolutely continuous with respect to Lebesgue measure.  After that, Trotter proved that the density $\ell $ of $ \mu_t $ is continuous in both time and space variables \cite{14} and, hence, coincides with the local time of the Brownian motion. Moreover, we get the following occupation formula  
$$      \int_{\mathbb{R}} \varphi(u)\ell(u,t)du = \int_{0}^{t}\varphi(w(s))ds ,\qquad a.s.  $$
which holds for every bounded and measurable function $\varphi $.

While thinking about local time for other Gaussian processes, S. Berman noticed that the independence of increments and self-similarity were extensively used in the Brownian motion case. Thus, he introduced the notion of local nondeterminism to remedy the lack of these properties for general Gaussian processes \cite{12}. Nonformally, the local nondeterminism signifies that the value of the process at a given time point is relatively unpredictable on the basis of a finite set of observations from the immediate past. S. Berman proved that this condition can assure the existence and smoothness of the local time. But the difficulty to check the local nondeterminism property was a motivation for A.A.Dorogovtsev to introduce a new class of Gaussian processes called by Gausssian integrators \cite{1}. These processes have a simple representation as Wiener integral where the integrand involves some bounded linear operator of the space $ L_2 ([0,1])$ of square integrable functions. This operator appears in the simple expression of the covariance and its invertibility is the key to local nondeterminism, existence and regularity of local time  \cite{7}. In addition, as their name indicates, stochastic integration with respect to integrators can be made. 
The used here integral is the Skorokhod ( or extended ) integral which generalizes the It\^{o}-integral, but does not keep the isometry property. In fact, different integrands can have the same integral. This property makes it possible to obtain different integral representations of the local time of integrators, which is our main purpose. The existence of such representations comes from the Clark formula. This formula states that every square integrable random variable $\alpha $ measurable with respect to the Wiener process $w$ can be uniquely written as 
$$ \alpha = \mathbb{E} \alpha + \int_0^1 x(t)dw(t) $$
where $x$ is a square integrable random element of $ L_2 ([0,1])$ adapted to the Wiener filtration. When the random variable $ \alpha $ is stochastically differentiable, then the process $x$ can be expressed in terms of $\alpha$ through the Clark-Ocone formula \cite{10}. In general case, we prefer obtaining an analogous formula where the integration is made with respect to the integrator itself. This pushs us to use the extended stochastic integral. And, as we mentioned before, we know that the same Skorokhod integral can be obtained from different integrands so that we predict it is possible to state more than one Clark formula for the local time of integrators.

This paper is divided into five sections. First we recall some vocabulary of white noise analysis emphasizing on chaotic decomposition and extended stochastic integral. Second section is devoted to the Clark formula where we state the result in terms of Skorokhod integral and find an optimal (in some sense ) representation called by the minimal norm integral representation. In the third section, Gaussian integrators are introduced, followed by some examples and a characterisation in terms of continuous linear operators and white noise. We devote section 4 to the chaotic expansion of the local time of integrators and the last one to some of its integral representations. 

\section{Preliminary facts from white noise analysis}

In this section we recall some notions of stochastic analysis that will be used along this paper. We focus on the concepts of white noise, chaotic expansion, stochastic derivative, extended stochastic integral and second quantization operators.

Let $H$ be a real separable Hilbert space with the scalar product $(\cdot,\cdot)$ and the norm $|\cdot|.$ We recall the definition of the white noise in $H$ and a few related notions. For the detailed exposition of the theory we refer to \cite{3, 4, 9, 10}. Through all the sections of this paper, we will use the probability space $(\Omega,\mathcal{F},\mathbb{P}).$ 
\begin{flushleft}
By definition, the white noise $\xi$ in $H$ is the collection $ \{(\xi,h),h\in H\} $ of random variables that satisfies following two properties 

\begin{enumerate}

\item[1)] each $(\xi,h)$ is a Gaussian random variable with mean $0$ and variance $|h|^2$,

\item[2)] $h\to (\xi,h)$ is a linear mapping from $H$ into $L^2(\Omega,\mathcal{F},\mathbb{P}).$

\end{enumerate}

The space $L^2(\Omega,\sigma(\xi),\mathbb{P})$ of all square integrable functionals from $\xi$ possess an orthogonal structure, known as the It\^o-Wiener expansion (or, chaos expansion). Denote by $H^{\otimes n}_{s}$ the space of all $n$-linear symmetric Hilbert-Schmidt forms $A_n:H^n\to \mbR.$ $H^{\otimes n}_{s}$ is a separable Hilbert space relatively to the scalar product
$$
(A_n,B_n)_n = \sum_{k_1,\ldots,k_n}A_n(e_{k_1},\ldots,e_{k_n})B_n(e_{k_1},\ldots,e_{k_n}),
$$
where $(e_k)$ is an orthonormal basis in $H.$ By $\|\cdot\|_n$ we denote norm in $H^{\otimes n}_s.$ Each form $A_n\in H^{\otimes n}_s$ can be written as a series
$$
A_n=\sum_{k_1,\ldots,k_n}a_{k_1,\ldots,k_n}e_{k_1}\otimes\ldots \otimes e_{k_n}.
$$
Associate to it a random variable 
$$
A_n(\xi,\ldots,\xi)=\sum_{j_1+j_2+\ldots=n}b_{j_1,j_2,\ldots}\prod^\infty_{l=1}H_{j_l}((\xi,e_l))\in L^2(\Omega,\sigma(\xi),\mathbb{P}).
$$
Here $H_k$ is the $k-$th Hermite polynomial $  H_k(x)=(-1)^k e^{ \frac{x^2}{2}}  \left(\frac{d}{dx}\right) ^{k} e^{ -\frac{x^2}{2}}  ,$ and the coefficient $b_{j_1,j_2,\ldots}$ equals 
$$
b_{j_1,j_2,\ldots}=\frac{n!}{j_1!j_2!\ldots}a_{\underbrace{1,\ldots,1}_{j_1}, \underbrace{2,\ldots,2}_{j_2},\ldots}.
$$ 
The random variable  $A_n(\xi,\ldots,\xi)$ is an (infinite-dimensional) Hermite polynomial from $\xi$ of a degree $n.$ Polynomials $A_n(\xi,\ldots,\xi)$ and $B_k(\xi,\ldots,\xi)$ of different degrees are orthogonal. Up to a constant, the correspondence
$$
A_n\to A_n(\xi,\ldots,\xi)
$$
is an isometry
$$
\mathbb{E} A_n(\xi,\ldots,\xi)^2=n! \|A_n\|^2_n.
$$
 The It\^o-Wiener expansion is the representation of $L^2(\Omega,\sigma(\xi),\mathbb{P})$ as an orthogonal sum of spaces of orthogonal polynomials. Each square integrable random variable $\alpha\in L^2(\Omega,\sigma(\xi),\mathbb{P})$ can be uniquely written as a sum
\begin{equation}
\label{x}
\alpha=\sum^\infty_{n=0}A_n(\xi,\ldots,\xi) , \ A_n\in H^{\otimes n}_s.
\end{equation}

The It\^o-Wiener expansion gives one possible way to define stochastic derivative. Consider the random variable $\alpha\in L^2(\Omega,\sigma(\xi),\mathbb{P})$ with the It\^o-Wiener expansion \eqref{x} . 
\begin{defn} \cite{15}
 We say that $\alpha $ is stochastically differentiable, if there exists square integrable $H-$ valued random element $D\alpha\in L^2(\Omega,\sigma(\xi),\mathbb{P};H),$ such that for every $h\in H$
$$
(D\alpha,h)=\sum^\infty_{n=0}nA_n(h,\underbrace{\xi,\ldots,\xi}_{n-1}).
$$
\end{defn}
The operator $D$ is an unbounded closed operator, acting from  $\mbox{Dom}(D)\subset L^2(\Omega,\sigma(\xi),\mathbb{P})$ into $L^2(\Omega,\sigma(\xi),\mathbb{P};H).$ 
\begin{defn} \cite{15}
The extended stochastic integral is the adjoint operator $I=D^*$ (in \cite{10} it is called the divergence operator).
\end{defn}
The operator $I$ is also an unbounded closed operator \cite[\S 1.3]{10}.
\end{flushleft}

Now consider the Wiener process $(w(t))_{t\in[0,1]},$ $w(0)=0.$ It naturally defines a white noise $\xi$ in the Hilbert space $L_2([0,1])$ by the formula
$$ (h, \xi ) = \int_0^1 h(t) dw(t). $$
  Informally, $\xi$ is the derivative of $w.$ Now, $\sigma(\xi)=\sigma(w).$ In this case, if a square integrable random process $y$ belongs to the domain of the extended stochasic integral $I,$ then $I(y)$ can simply be denoted by $\displaystyle \int_0^1 y(t)dw(t) .$ We recall that if the process $y$ is adapted to the Wiener filtration then the integral $I(y)$ coincides with the It\^{o} integral \cite{15}.\\
 
  Now let us finish this part by recalling the definition of a second quantization operator.  Using the white noise $\xi $ generated by the Brownian motion $(w(t))_{t\in[0,1]}$, every square integrable random variable $\alpha$ measurable with respect to $\xi$ can be uniquely represented as a sum
 $$
 \alpha=\sum^\infty_{n=0}A_n(\xi, \ldots, \xi)=\sum^\infty_{n=0}\int_{\Delta_n}a_n(t_1, \ldots, t_n)dw(t_1)\ldots dw(t_n),
 $$
 where $ \displaystyle \Delta_{n} = \left\lbrace ( t_{1},\ldots,t_{n}), \; 0 \leqslant t_{1} \leqslant \ldots \leqslant t_{n} \leqslant 1 \right\rbrace $ and $A_n(\xi, \ldots, \xi)$ is a Hilbert-Shmidt form from $\xi$ or the multiple Wiener integrals from the nonrandom kernels $a_n.$ For continuous linear operator $B$ in $L_2([0; 1])$ define new Hilbert--Shmidt form
 $$
 A^B_n(\vf_1, \ldots, \vf_n)=A_n(B^*\vf_1, \ldots, B^*\vf_n).
 $$
 \begin{defn} \cite{5} If the operator norm $\|B\|\leq1,$ then the sum
 $$
 \sum^\infty_{n=0}A^B_n(\xi, \ldots, \xi)
 $$
 converges in square-mean and is an action of the second quantization operator $\Gamma(B)$ on $\alpha.$ 
 \end{defn}
 The very simple form $\Gamma(B)$ has on the stochastic exponents, i.e. on the random variables of the form $\cE(h)=e^{(h, \xi)-\frac{1}{2}\|h\|^2}, h\in L_2([0, 1]).$ It can be easily checked \cite{5} that
 $$
 \Gamma(B^*)\cE(h)=\cE(Bh).
 $$
 The operators of second quantization are the partial case of conditional expectation \cite{5, 6} and have all its properties.

\section{Clark formula}
 
In this section we recall the Clark representation formula for square integrable functionals from $w$ (the Wiener process ). Then, we propose a generalisation in terms of extended stochastic integral and we precise an optimal representation called by the minimal norm integral representation.

It is known \cite[Th. 1.1.3]{10} that for each $\alpha\in L^2(\Omega,\sigma(w),\mathbb{P})$ there exists unique square integrable $w-$predictable process $(u(t))_{t\in [0,1]},$ such that the following relation with the It\^{o} integral holds 
\begin{equation}
\label{z}
\alpha=\mathbb{E}\alpha +\int^1_0 u(t)dw(t),\quad  \mathbb{E} \alpha ^2 = (\mathbb{E} \alpha )^2 + \mathbb{E} \int_0^1 u(t)^2 dt .
\end{equation}

As an example, let us derive the Clark representation for one-dimensional functional
$$
\alpha=f(w(1)).
$$
Of course, the formula is not new (see, for example, \cite{11}). Still we present the proof, because the only assumption we make about $f$ is the square integrability. Let $p_s(x)$ be the transition density of the Wiener process 
$$
p_s(x)=(2\pi s)^{-1/2}\exp(-x^2/2s).
$$
Denote by $(P_s)_{s\geq 0}$ the transition semigroup of the Wiener  process
$$
P_sf(x)=\mathbb{E}f(x+w(s))=\int_{\mathbb{R}}f(y)p_s(x-y)dy.
$$

\begin{lem}
\label{lem3}
For every measurable function $f$ with $\displaystyle \int_{\mathbb{R}}f^2(y)p_1(y)dy<\infty,$ the Clark representation for $f(w(1))$ is 
\begin{equation}
\label{lem3_1}
f(w(1))=\mathbb{E} f(w(1)) +\int^1_0 \partial_x P_{1-t}f(w(t))dw(t).
\end{equation}
\end{lem}

\begin{proof}
Note that for each compact set $K\subset (0,1)\times \mathbb{R}$ and all $n,m\geq 0$ one has
\begin{equation}
\label{aaa}
\sup_{y\in \mathbb{R}} \sup_{(s,x)\in K}\bigg|  \frac{\partial^n_s \partial^m_x p_s(x-y)}{p_1(y)}  \bigg|<\infty.
\end{equation}
Hence, the function $P_sf(x)$ is correctly defined for $0<s<1, $ $x\in \mathbb{R}$ and is smooth. In fact,
\begin{equation}
\label{lem3_2}
\partial^n_s \partial^m_x P_sf(x)=\int_{\mathbb{R}}f(y) \partial^n_s \partial^m_x p_s(x-y) dy.
\end{equation}
Consequently, 
\begin{equation}
\label{lem3_3}
\partial_s P_sf(x)=\frac{1}{2}\partial^2_x P_sf(x), \ 0<s<1, x\in\mathbb{R}.
\end{equation}

It is enough to prove \eqref{lem3_1} only for continuous compactly supported functions. Indeed, given any measurable $f$ with $\displaystyle \int_{\mathbb{R}}f^2(y)p_1(y)dy<\infty,$ the Clark representation theorem implies that there exists square integrable $w-$predictable process $(u(t))_{t\in [0,1]},$ such that
\begin{equation}
\label{lem3_4}
f(w(1))-\mathbb{E} f(w(1)) =\int^1_0 u(t)dw(t).
\end{equation}
On the other hand, there exists sequence $(f_n)_{n\geq 1}$ of continuous compactly supported functions, such that
$$
\int_{\mathbb{R}}(f_n(y)-f(y))^2p_1(y)dy\to 0, \ n\to\infty.
$$
If \eqref{lem3_1} is proved for $f_n,$ then
$$
f_n(w(1))-\mathbb{E} f_n(w(1)) =\int^1_0 \partial_x P_{1-t}f_n(w(t))dw(t).
$$
The left-hand side of the latter equality converges to the left-hand side of \eqref{lem3_4} in $L^2$. Correspondingly, the right-hand side also converges, i.e.
$$
\int^1_0\mathbb{E}(\partial_x P_{1-t}f_n(w(t))-u(t))^2dt\to 0, n\to\infty.
$$
From \eqref{lem3_2} it follows that pointwisely
$$
\partial_x P_{1-t}f_n(w(t))\to \partial_x P_{1-t}f(w(t)), n \to\infty.
$$
Hence, the needed representation for $f$ follows.
It remains to check the case when $f$ is continuous and compactly supported. Applying the It\^o formula to the function $(t,x)\to P_{1-t}f(x)$ and the process $(w(t))_{\varepsilon\leq t\leq 1-\varepsilon},$ and using \eqref{lem3_3} one obtains the representation
$$
P_\varepsilon f(w(1-\varepsilon))=P_{1-\varepsilon} f(w(\varepsilon))+\int^{1-\varepsilon}_\varepsilon \partial_x P_{1-t}f(w(t))dw(t).
$$
Additional assumptions on $f$ imply
$$
P_\varepsilon f(w(1-\varepsilon))\to f(w(1)), \ \varepsilon\to 0,
$$ 
and
$$
P_{1-\varepsilon} f(w(\varepsilon))\to P_1f(0)=\mathbb{E}f(w(1)), \ \varepsilon\to 0.
$$
The lemma is proved.

\end{proof}

\begin{remark}
\label{rem2}
Using the equation \eqref{lem3_2}, the formula \eqref{lem3_1} can be rewritten as
$$
f(w(1))=\mathbb{E} f(w(1)) +\int^1_0\bigg( \int_{\mathbb{R}}f(y)\partial_x p_{1-t}(w(t)-y)dy\bigg)dw(t).
$$
\end{remark}

Now we return to the genaral situation. Let $\xi$ be a white noise in the real separable Hilbert space $H$.  All the random variables are assumed to be measurable with respect to $\xi,$ i.e. $\mathcal{F}=\sigma(\xi).$ Then any square integrable random variable can be written in the form analogous to \eqref{z}.

\begin{lem}
\label{lem1}
Any $\alpha\in L^2(\Omega,\mathcal{F},\mathbb{P})$  can be written in the form
$$
\alpha=\mathbb{E} \alpha +I u,
$$
for some $u$ from the domain of $I.$
\end{lem}
 In fact, we prove the lemma \ref{lem1} by constructing a bounded linear operator
$$
d:L^2(\Omega,\mathcal{F},\mathbb{P})\to L^2(\Omega,\mathcal{F},\mathbb{P};H),
$$
such that 
$$
\alpha=\mathbb{E}\alpha+I(d\alpha).
$$ 
 
\begin{proof} 
Consider an It\^o-Wiener expansion of $\alpha $
$$
\alpha=\mathbb{E}\alpha + \sum^\infty_{n=1}A_n(\xi,\ldots,\xi),
$$
and define $d\alpha\in L^2(\Omega,\mathcal{F},\mathbb{P};H) $ as follows
$$
d\alpha =\sum^\infty_{n=1} A_n(\cdot,\underbrace{\xi,\ldots,\xi}_{n-1}).
$$ 
Here, $A_n$ is viewed as an $(n-1)-$linear $H-$valued Hilbert-Schmidt form on $H.$ Respectively,  $A_n(\cdot,\underbrace{\xi,\ldots,\xi}_{n-1})$ is an $H-$valued orthogonal polynomial from $\xi$ of degree $n-1.$ The series in the definition of $d\alpha$ converges, as 
$$
\sum^\infty_{n=1}\mathbb{E} |A_n(\cdot,\xi,\ldots,\xi)|^2=\sum^\infty_{n=1} (n-1)! \|A_n\|^2_{n}
\leq \sum^\infty_{n=0} n! \|A_n\|^2_{n}=\mathbb{E}\alpha^2.
$$
In order to check, that $\alpha=\mathbb{E}\alpha +I(d\alpha),$  consider arbitrary stochastically differentiable function $\beta\in L^2(\Omega,\mathcal{F},\mathbb{P}).$ Write the It\^o-Wiener expansion of $\beta $
$$
\beta=\sum^\infty_{n=0}B_n(\xi,\ldots,\xi).
$$
Then, 
$$
\mathbb{E} (D\beta, d\alpha)=\mathbb{E} \sum^\infty_{n=1}n( A_n(\cdot,\xi,\ldots,\xi),B_n(\cdot,\xi,\ldots,\xi))=\sum^\infty_{n=1} n! (A_n,B_n)_{n}=
\mathbb{E} (\alpha-\mathbb{E}\alpha)\beta.
$$
The needed equality is verified.
\end{proof}

As a corollary of the lemma \ref{lem1}, to each $\alpha\in L^2(\Omega,\mathcal{F},\mathbb{P})$ we've associated a non-empty set
$$
\mathcal{K}(\alpha)=\{u\in \mbox{Dom}(I): Iu=\alpha-\mathbb{E}\alpha\}.
$$
It is evident, that $\mathcal{K}(\alpha)$ is an affine subspace. It is closed in $L^2(\Omega,\mathcal{F},\mathbb{P};H),$ because the operator $I$ is closed. Consequently, the norm attains its minimum on a unique element $u(\alpha)\in \mathcal{K}(\alpha).$ In other words, for each $\alpha\in L^2(\Omega,\mathcal{F},\mathbb{P})$ there exists unique $u(\alpha)\in L^2(\Omega,\mathcal{F},\mathbb{P};H),$ such that  
$$
u(\alpha)=\mbox{argmin}_{u\in \mathcal{K}(\alpha)}\mathbb{E}|u|^2.
$$

We will refer to the expression 
$$
\alpha=\mathbb{E}\alpha+I(u(\alpha))
$$
as to the minimal norm integral representation of $\alpha$.

Next we prove that the minimal norm integrand $u(\alpha)$ coincides with $d\alpha$.
\begin{lem}
\label{lem2}
Let $\alpha\in L^2(\Omega,\mathcal{F},\mathbb{P}).$ Then 
$$
u(\alpha)=d\alpha.
$$
\end{lem}

\begin{proof}
Recall from the lemma \ref{lem1}, that 
$$
\mathbb{E}|d\alpha |^2=\sum^\infty_{n=1}(n-1)!\|A_n\|^2_n.
$$
To show that $u(\alpha)=d\alpha$ we need to check that for each $u\in \mathcal{K}(\alpha),$ $\mathbb{E}|u|^2 \geq \mathbb{E}|d\alpha|^2.$ Consider the It\^o-Wiener expansion of $u$. For any $h\in H,$ 
$$
(u,h)=\sum^\infty_{n=1}B_{n}(\underbrace{\xi,\ldots,\xi}_{n-1};h),
$$
where $B_{n}$ is $n-$linear Hilbert-Schmidt form on $H,$ that is symmetric in first $n-1$ arguments. Let us calculate $\mathbb{E}|u|^2.$
$$
\mathbb{E}|u|^2=\sum_k \mathbb{E}\bigg(\sum^\infty_{n=1} B_{n}(\underbrace{\xi,\ldots,\xi}_{n-1};e_k)\bigg)^2=
$$
$$
=\sum_k \sum^\infty_{n=1}(n-1)! \| B_{n}(\cdot;e_k)\|^2_{n-1}=\sum^\infty_{n=1}(n-1)!\|B_n\|^2_n.
$$
The action of $I$ on $u$ is expressed in terms of the It\^o-Wiener expansion in the following way
$$
Iu=\sum^\infty_{n=1}\hat{B_{n}}(\underbrace{\xi,\ldots,\xi}_{n}),
$$ 
where $\hat{B_{n}}$ denotes the symmetrization of $B_{n}.$ But, $Iu=\alpha-\mathbb{E}\alpha,$ so
$$
A_n=\hat{B_n}.
$$
It is well-known that the symmetrization operator has the norm $1,$ i.e.
$$
\|A_n\|_n\leq \|B_n\|_n,
$$
and $\mathbb{E}|u|^2\geq \mathbb{E}|d\alpha|^2.$ 
\end{proof}

In the next lemma we derive analytical representation of the operator $d.$ Below $(T_t)_{t\geq 0}$ denotes the Ornstein-Uhlenbeck semigroup \cite[\S 1.4]{10}. 

\begin{lem}
\label{lem4}
For any $\alpha\in L^2(\Omega,\mathcal{F},\mathbb{P})$ the following relation holds
\begin{equation}
\label{30_07}
d\alpha=D\bigg(\int^\infty_0 T_t(\alpha-\mathbb{E}\alpha)dt\bigg)=\int^\infty_0 D(T_t \alpha) dt,
\end{equation}
where integrals are taken in the Bochner's sense.
\end{lem}

\begin{proof}
Recall that in terms of the It\^o-Wiener expansion
$$
T_t(\alpha-\mathbb{E}\alpha)=\sum^\infty_{n=1} e^{-nt}A_n(\xi,\ldots,\xi),
$$
The norm of $T_t(\alpha-\mathbb{E}\alpha)$ equals
$$
\|T_t(\alpha-\mathbb{E}\alpha)\|=\sqrt{\sum^\infty_{n=1} e^{-2nt}n!\|A_n\|^2_n}\leq
e^{-t}\sqrt{\mathbb{E}(\alpha-\mathbb{E}\alpha)^2}.
$$
Hence, the Bochner integral $\int^\infty_0 T_t(\alpha-\mathbb{E}\alpha) dt$ is well-defined. Its It\^o-Wiener expansion is
$$
\int^\infty_0 T_t(\alpha-\mathbb{E}\alpha) dt=\sum^\infty_{n=1}\frac{1}{n}A_n(\xi,\ldots,\xi),
$$
so the first equality \eqref{30_07} follows. The second one follows from the fact that the stochastic derivative $D$ is the closed operator.
\end{proof}

The next result is the minimal norm integral representation of the functional 
$$
\alpha=f((g,\xi))
$$ 
in the general case, i.e. when $\xi$ is a white noise in the Hilbert space $H,$ $g\in H.$ Denote by $\Phi(x)$ the distribution function of the standard Gaussian distribution
$$
\Phi(x)=\frac{1}{\sqrt{2\pi}}\int^x_{-\infty}e^{-\frac{\tau^2}{2}}d\tau.
$$ 

\begin{lem}
\label{lem5}
Consider random variable $\alpha=f((g,\xi))\in L^2(\Omega,\mathcal{F},\mathbb{P}).$ Then
$$
d\alpha=\bigg(\int_{\mathbb{R}}f(y)\frac{1}{|g|^2}e^{\frac{(g,\xi)^2-y^2}{2|g|^2}}\bigg(\Phi\bigg(\frac{(g,\xi)}{|g|}\bigg)-1_{(g,\xi)>y}\bigg)dy\bigg)g.
$$
Respectively, its minimal norm integral representation is
$$
f((g,\xi))=\mathbb{E}f((g,\xi))+I\bigg[\bigg(\int_{\mathbb{R}}f(y)\frac{1}{|g|^2}e^{\frac{(g,\xi)^2-y^2}{2|g|^2}}\bigg(\Phi\bigg(\frac{(g,\xi)}{|g|}\bigg)-1_{(g,\xi)>y}\bigg)dy\bigg)g\bigg].
$$
\end{lem}

\begin{proof}

According to the lemma \ref{lem4},
$$
d\alpha =\int^\infty_0 D(T_t \alpha)dt.
$$
We will use well-known integral representation for $T_t\alpha$  \cite[\S 1.4]{10}
\begin{equation}
\label{lem5_1}
T_t\alpha =\mathbb{E}[f(e^{-t}(g,\xi)+\sqrt{1-e^{-2t}}(g,\xi'))/\xi],
\end{equation}
where $\xi'$ is a white noise in $H,$  independent from $\xi.$ Equivalently, 
$$
T_t\alpha =h((g,\xi)),
$$
where
$$
h(x)=\int_{\mathbb{R}} f(y)p_{(1-e^{-2t})|g|^2}(e^{-t}x-y)dy.
$$
From \eqref{aaa} it follows that
$$
h'(x)=\frac{e^{-t}}{(1-e^{-2t})|g|^2}\int_{\mathbb{R}}f(y)(y-e^{-t}x)p_{(1-e^{-2t})|g|^2}(e^{-t}x-y) dy.
$$
Using the representation \eqref{lem5_1}, $h'((g,\xi))$ can be written as
$$
h'((g,\xi))=\frac{e^{-t}}{\sqrt{1-e^{-2t}}|g|^2}\mathbb{E}[f(e^{-t}(g,\xi)+\sqrt{1-e^{-2t}}(g,\xi'))(g,\xi')/\xi].
$$
It follows that $\mathbb{E}(h'((g,\xi)))^2<\infty,$ and the stochastic derivative of $T_t\alpha$ equals
$$
D(T_t\alpha)=\frac{e^{-t}}{(1-e^{-2t})|g|^2}\bigg(\int_{\mathbb{R}}f(y)(y-e^{-t}(g,\xi))p_{(1-e^{-2t})|g|^2}(e^{-t}(g,\xi)-y) dy\bigg) g.
$$
Hence,
$$
d\alpha=\bigg( \int^\infty_0  \frac{e^{-t}}{(1-e^{-2t})|g|^2}\int_{\mathbb{R}}f(y)(y-e^{-t}(g,\xi))p_{(1-e^{-2t})|g|^2}(e^{-t}(g,\xi)-y) dy dt\bigg) g.
$$
The result of lemma \ref{lem4}  is not enough to interchage the order of integration. However, note that
$$
\mathbb{E}\int^\infty_0  \frac{e^{-t}}{(1-e^{-2t})|g|^2}\int_{\mathbb{R}}\big|f(y)(y-e^{-t}(g,\xi))\big|p_{(1-e^{-2t})|g|^2}(e^{-t}(g,\xi)-y) dy dt=
$$
$$
=\int^\infty_0 \frac{e^{-t}}{\sqrt{1-e^{-2t}}|g|^2} \mathbb{E} \mathbb{E} \big[\big|f(e^{-t}(g,\xi)+\sqrt{1-e^{-2t}}(g,\xi'))(g,\xi')\big|/\xi\big] dt=
$$
$$
=\int^\infty_0 \frac{e^{-t}}{\sqrt{1-e^{-2t}}|g|^2} \mathbb{E} \big|f(e^{-t}(g,\xi)+\sqrt{1-e^{-2t}}(g,\xi'))(g,\xi')\big| dt\leq
$$
$$
\leq \mathbb{E}(f((g,\xi)))^2\int^\infty_0 \frac{e^{-t}}{\sqrt{1-e^{-2t}}}dt <\infty.
$$
So, $d\alpha$ can be written as
$$
d\alpha=\bigg(\int_{\mathbb{R}}f(y)\frac{1}{|g|^2} \Delta_{\frac{y}{|g|}}\bigg(\frac{(g,\xi)}{|g|}\bigg)  dy \bigg) g,
$$
where
$$
\Delta_y(x)=\int^\infty_0  \frac{e^{-t}}{(1-e^{-2t})} (y-e^{-t}x)p_{(1-e^{-2t})}(e^{-t}x-y) dt=
$$
\begin{equation}
\label{02_09}
=\frac{1}{\sqrt{2\pi}}\int^\infty_0 \frac{e^{-t}}{(1-e^{-2t})^{3/2}} (y-e^{-t}x) e^{-\frac{(y-e^{-t}x)^2}{2(1-e^{-2t})}} dt.
\end{equation}
Hence, it is enough to check that the equality
\begin{equation}
\label{03_09}
\Delta_y(x)=e^{\frac{x^2-y^2}{2}}(\Phi(x)-1_{x>y})
\end{equation}
holds for a.a. $y\in \mathbb{R}.$ Consider $y\ne -x,x.$ The change of variables
$$
\tau=\sqrt{\frac{1+e^{-t}}{1-e^{-t}}}\frac{|y-x|}{2}-\sqrt{\frac{1-e^{-t}}{1+e^{-t}}}\frac{|y+x|}{2}
$$
leads to the formula
$$
\Delta_y(x)=\frac{1}{\sqrt{2\pi}} e^{-\frac{|y^2-x^2|+y^2-x^2}{2}}\frac{y+x}{|y+x|}
$$
$$
\int^\infty_{\frac{|y-x|-|y+x|}{2}}\frac{\frac{|y^2-x^2|+y^2-x^2}{2}+\tau^2-\tau\sqrt{\tau^2+|y^2-x^2|}}{|y^2-x^2|+\tau^2-\tau\sqrt{\tau^2+|y^2-x^2|}}e^{-\frac{\tau^2}{2}}d\tau.
$$
In the case $|y|>|x|$ one has
$$
\Delta_y(x)=\frac{y+x}{|y+x|}e^{\frac{x^2-y^2}{2}}\frac{1}{\sqrt{2\pi}}\int^\infty_{\frac{|y-x|-|y+x|}{2}}e^{-\frac{\tau^2}{2}}d\tau.
$$
In the case $|y|<|x|$ one has
$$
\Delta_y(x)=-\frac{y+x}{|y+x|}\frac{1}{\sqrt{2\pi}}\int^\infty_{\frac{|y-x|-|y+x|}{2}}\frac{\tau}{\sqrt{\tau^2+x^2-y^2}}e^{-\frac{\tau^2}{2}}d\tau=
$$
$$
=-\frac{y+x}{|y+x|}e^{\frac{x^2-y^2}{2}}\frac{1}{\sqrt{2\pi}}\int^\infty_{|x|}e^{-\frac{\tau^2}{2}}d\tau.
$$
Now the equality \eqref{03_09} is easily checked.
\end{proof}

\begin{example}
\label{ex03_09}
In the case when the white noise $\xi$ is generated by the Wiener process $(w(t))_{t\in [0,1]},$ lemmas \ref{lem3} and \ref{lem5} give two different representations for the random variable $f(w(1))\in L^2(\Omega,\mathcal{F},\mathbb{P})$ (see also remark \ref{rem2})
\begin{itemize}
\item
the Clark representation
$$
f(w(1))=\mathbb{E} f(w(1)) +\int^1_0\bigg(\int_{\mathbb{R}}f(y) \partial_x p_{1-t}(w(t)-y)dy\bigg)dw(t);
$$
\item 
the minimal norm integral representation
$$
f(w(1)) =\mathbb{E}f(w(1)) +\int^1_0\bigg(\int_{\mathbb{R}}f(y)e^{\frac{w(1)^2-y^2}{2}}(\Phi(w(1))-1_{w(1)>y})dy\bigg)  dw(t).
$$
\end{itemize}
Notice that the integral with respect to the Wiener process in the first representation is the It\^o stochastic integral and in the second one it is the extended stochastic integral.
\end{example}

\section{Gaussian integrators }

In this section, we recall the definition of Gaussian integrators and give some examples. We also find a formula that represents them explicitly as white noise functionals.

 \begin{defn} \cite{1}
 A centered Gaussian process $x(t),\ t\in[0; 1]$ is said to be an integrator if there exists a constant $c>0$ such that for an arbitrary partition $0=t_0<t_1<\ldots<t_n=1$ and real numbers $a_0, \ldots, a_{n-1} $
 \begin{equation}
 \label{a}
 \mathbb{E}\Big(\sum^{n-1}_{k=0}a_k(x(t_{k+1})-x(t_k))\Big)^2\leq c\sum^{n-1}_{k=0}a^2_k\Delta t_k.
 \end{equation}
 \end{defn}
 \begin{example} {\bf Wiener process}
$$ x(t)=w(t),\ t\in[0;1] $$
Really,
$$
 \mathbb{E}\Big(\sum^{n-1}_{k=0}a_k(w(t_{k+1})-w(t_k))\Big)^2= \sum^{n-1}_{k=0}a^2_k\Delta t_k.
$$
\end{example}
 \begin{example} {\bf Brownian bridge}
$$
x(t)=w(t)-tw(1),\ t\in[0;1]
$$
Let us check inequality \eqref{a} for the Brownian bridge. One can see that

$$
\mathbb{E}\Big(\sum^{n-1}_{k=0}a_k(w(t_{k+1})-w(t_k)-(t_{k+1}-t_k)w(1)\Big)^2\leq
$$
$$
\leq 2\mathbb{E}\Big(\sum^{n-1}_{k=0}a_k(w(t_{k+1})-w(t_k)\Big)^2+2\mathbb{E}\Big(\sum^{n-1}_{k=0}a_k(t_{k+1}-t_k)w(1)\Big)^2=
$$
$$
=2\sum^{n-1}_{k=0}a^2_k\Delta t_k+2(\sum^{n-1}_{k=0}a_k\Delta t_k)^2\leq
$$
$$
\leq 4\sum^{n-1}_{k=0}a^2_k\Delta t_k.
$$
\end{example}
 \begin{example} {\bf Fractional Brownian motion}
$$
x(t)=B^\alpha_t,\  \alpha\in(0;1),\ t\in[0;1].
$$
Recall the definition. A zero mean Gaussian process with the
$$
\cov(B^\alpha_{t_1},B^\alpha_{t_2})=\frac{1}{2}(t^{2\alpha}_1+t^{2\alpha}_2-(t_2-t_1)^{2\alpha})
$$
is said to be a fractional Brownian motion, where $\alpha$ is called the Hurst index or Hurst parameter associated with the fractional Brownian motion.
Let us check whether the fractional Brownian motion $B^\alpha_t,\  \alpha\in(0;1),\ t\in[0;1]$ is an integrator. If $B^{\alpha}_t$ were integrator, then $c>0$ would exist such that
$$
\mathbb{E}(B^{\alpha}_{t_2}-B^{\alpha}_{t_1})^2=(t_2-t_1)^{2\alpha}\leq c(t_2-t_1).
$$
For $\alpha<\frac{1}{2}$
\begin{equation}
\label{b}
\lim_{t_2-t_1\to0}(t_2-t_1)^{2\alpha-1}=+\infty.
\end{equation}
 \eqref{b} implies that for $\alpha<\frac{1}{2}$ the fractional Brownian motion $B^{\alpha}_t$ is not an integrator. In the case of $\alpha=\frac{1}{2}$ the process $B^{\alpha}_t$ as a standard Brownian motion which is an integrator. Let us investigate the case $\alpha>\frac{1}{2}.$ To do that we will use the following statement.
\begin{lem}
The process $x$ is an integrator iff there exists $c>0$ such that for any two times continuously differentiable function $f$ on $[0; 1]$ with $f(0)=f(1)=0$ the following relation holds
\begin{equation}
\label{c}
\mathbb{E}\Big(\int^1_0x(t)f'(t)dt\Big)^2\leq c\int^1_0f^2(t)dt.
\end{equation}
\end{lem}
Let us check that for $\alpha>\frac{1}{2}$ the process $B^\alpha_t$ satisfies the inequality \eqref{c}. Really,
$$
\mathbb{E}\Big(\int^1_0x(t)f'(t)dt\Big)^2=\int^1_0\int^1_0f'(t_1)f'(t_2)(t_2-t_1)^{2\alpha}dt_2dt_1=
$$
$$
=4\alpha(2\alpha-1)\int^1_0f(t_1)\int^1_{t_1}f(t_2)(t_2-t_1)^{2\alpha-2}dt_2dt_1.
$$
To check that for $\alpha>\frac{1}{2}$ the integral operator in $L_2([0;1])$ with the kernel $K(t_1, t_2)=(t_2-t_1)^{2\alpha-2}\mathbf{1}_{\{t_2>t_1\}}$ is bounded one can use the Shur test.
\begin{thm}\cite{2} [Shur test]
If there exist positive functions $p, q: [0; 1]\to(0; +\infty)
$
and
$\alpha, \beta>0$  such that
$$
\int^1_{0}k(s_1, s_2)q(s_2)ds_2\leq \alpha p(s_1),
$$
$$
\int^1_{0}k(s_1, s_2)p(s_1)ds_1\leq \beta q(s_2),
$$
then $k$  corresponds to the  bounded operator with the norm less or equal to $\alpha\beta.$
\end{thm}
By using the Shur test one can see that for $\alpha>\frac{1}{2}$ the integral operator in $L_2([0;1])$ with the kernel $K(t_1, t_2)=(t_2-t_1)^{2\alpha-2}\mathbf{1}_{\{t_2>t_1\}}$ is bounded. It implies that for $\alpha>\frac{1}{2}$ the fractional Brownian motion $B^\alpha_t$ is an integrator.
\end{example}
The following two statements also give examples of integrators.
 \begin{lem}
 Suppose that Gaussian process $x$ is continuously differentiable almost surely. Then $x$ is an integrator.
 \end{lem}
\begin{proof}
Consider the process $x^{\prime}$. It is continuous Gaussian process which can be considered  as  a  Gaussian  random  element  in  $C ([0; 1])$. Consequently  the  uniform norm of $x^{\prime}$ is square integrable \cite{17}. Then
$$
\mathbb{E}\Big(\sum^{n-1}_{k=0}a_k(x(t_{k+1})-x(t_k))\Big)^2\leq
$$
$$
\leq \mathbb{E}(\sup_{[0;1]}x^{\prime})^2\sum^{n-1}_{k=0}a^2_k\Delta t_k
$$
\end{proof}
 \begin{lem}
 Let $\xi$ be a white noise in $L_2([0;1])$ and $L\subset L_2([0;1]).$ Denote by $\mathfrak{F}_{L}=\sigma((\varphi,\xi),\ \varphi\in L).$ Then $x(t)=\mathbb{E}(w(t)/\mathfrak{F}_L)$ is an integrator.
 \end{lem}
 \begin{proof}
Note that $ x(t), t\in [0,1] $ is Gaussian process. Let us check that it satisfies inequality \eqref{a}. Using Jensen inequality one can check that 
 $$
\mathbb{E}\Big(\sum^{n-1}_{k=0}a_k(x(t_{k+1})-x(t_k)\Big)^2 \leqslant  \mathbb{E}\Big(\sum^{n-1}_{k=0}a_k(\mathbb{E}(w(t_{k+1})-w(t_k)/\mathfrak{F}_L)\Big)^2\leq
$$
$$
\leq \mathbb{E}\left( \mathbb{E}\Big(\sum^{n-1}_{k=0}a_k(w(t_{k+1})-w(t_k)\Big)^2/\mathfrak{F}_L \right) =
$$
$$
=\mathbb{E}\Big(\sum^{n-1}_{k=0}a_k(w(t_{k+1})-w(t_k)\Big)^2 = \sum^{n-1}_{k=0}a^2_k\Delta t_k.
$$
 \end{proof}
The following statement describes the structure of integrators.

 \begin{prop}
 The centered Gaussian process $x(t),\ t\in[0; 1]$ is an integrator iff there exist Gaussian white noise $\xi$ in $L_2([0; 1])$ and continuous linear operator $A$ in the same space such that
 \begin{equation}
 \label{d}
x(t)=(A\mathbf{1}_{[0; t]}, \xi),\ t\in[0; 1].
\end{equation}
\end{prop}
\begin{proof}
Suppose that the Gaussian process $x$ has a representation \eqref{d}, then
$$
\mathbb{E}\Big(\sum^{n-1}_{k=0}a_k(x(t_{k+1})-x(t_k))\Big)^2=
\mathbb{E}\Big(\sum^{n-1}_{k=0}a_k(A\mathbf{1}_{[t_k; t_{k+1}]}, \xi)\Big)^2=
$$
$$
=\Big\|A\sum^{n-1}_{k=0}a_k\mathbf{1}_{[t_k; t_{k+1}]}\Big\|^2\leq
\|A\|^2\sum^{n-1}_{k=0}a^2_k(t_{k+1}-t_k).
$$
Inversely, suppose that $x$ is an integrator. Denote by $\ov{LS\{x\}}$ the closure of the linear span of values of $x.$ $\ov{LS\{x\}}$ is a separable Hilbert space with respect to square mean norm. It means that there exists an isomorphic embedding $j: \ \ov{LS\{x\}}\to L_2([0; 1]).$ Denote by $H_1=j(\ov{LS\{x\}}).$ Then $L_2([0; 1])$ can be represented as a direct sum $H_1\oplus H_2.$ Suppose that $\xi_2$ is a Gaussian white noise in $H_2,$ which is independent of $x.$ Put $(h_1, \xi_1)=j^{-1}(h_1)$ and $(h, \xi)=(h_1, \xi_1)+(h_2, \xi_2),\ h_1\in H_1,\  h_2\in H_2,\ h\in L_2([0; 1]).$ The independence of $\xi_2$ and $x$ implies that $\xi$ is a white  noise in $L_2([0; 1]).$ Define a continuous linear operator $B: L_2([0; 1])\to\ov{LS\{x\}}$ by the rule
$$
B\Big(\sum^{n-1}_{k=0}a_k\mathbf{1}_{[t_k; t_{k+1}]}\Big)\mapsto\sum^{n-1}_{k=0}a_k(x(t_{k+1})-x(t_k)).
$$
Then for an operator $A=jB$
$$
B\mathbf{1}_{[0; t]}=x(t)=(jx(t), \xi)=(A\mathbf{1}_{[0; t]}, \xi).
$$
Since $x$ satisfies \eqref{a}, for any step function $f\in L_2[0; 1]$ the next inequality holds

\begin{equation}
\label{e}
\|Af\|^2\leq c\|f\|^2.
\end{equation}

The set of all step functions on $[0; 1]$ is dense in $L_2([0; 1]).$ Consequently \eqref{e} ends the proof.
\end{proof}
{\bf Examples}
\begin{enumerate}
\item  Wiener process $A=I$
\item  Brownian bridge $A=I-P,$ where $P$ is a projection on $\mathbf{1}_{[0;1]}$
\item  Fractional Brownian motion for $\alpha>\frac{1}{2}$, $A$ is the integral operator in $L_2([0;1])$ with the kernel $K(t_1, t_2)=(t_2-t_1)^{2\alpha-2}\mathbf{1}_{\{t_2>t_1\}}$.
\end{enumerate}

We end this section by introducing the notion of extended stochastic integral with respect to integrators and recalling one of its useful properties. For this let $\xi$ be a Gaussian white noise in $L_2([0; 1]).$  $\xi$ can be considered as a formal derivative of the Wiener process $w(t)=(\mathbf{1}_{[0, t]}, \xi), t\in[0, 1].$ In \cite{1,5} the extended stochastic integral for random function $y$ from $L_2([0; 1])$ was defined as a regularized product $(y, \xi)$ which is denoted also by $ \displaystyle \int^1_0y(s)dw(s).$
 In such terms the extended stochastic integral with respect to the integrator $x$ which has the representation 
\eqref{d} can be defined as follows
 $$
 \int^1_0y(s)dx(s):=(Ay, \xi)=\int^1_0(Ay)(s)dw(s).
 $$
 It is a part of the general framework on action of Gaussian random operator on random elements \cite{3, 4}. In details the stochastic calculus for integrators was considered in \cite{1, 5}. Here we need the following important relationship between the extended stochastic integral and operators of second quantization. We consider, for this purpose, a continuous linear operator $B$ in $L_2([0,1])$ such that $ \|B\| \; \leqslant \; 1 $. Then, we have 
 \begin{thm} \label{1}{\rm \cite{5}}
 Suppose that the square-integrable random element $y$ in $L_2([0; 1])$ belongs to the domain of the extended stochastic integral. Then $\Gamma(B)y$ belongs to the domain of the extended stochastic integral with respect to the integrator $x=\Gamma(B)w,$ i.e.
 $$
 \Gamma(B)\int^1_0y(t)dw(t)=\int^1_0(\Gamma(B)y)(t)dx(t).
 $$
 \end{thm}
 This statement gives a possibility in \cite{5} to obtain the anticipating stochastic PDE for second quantization of the certain functionals from the diffusion processes. In this paper, we will apply Theorem 12 in order to obtain the integral representation for the local time of the one-dimensional integrator ( see Theorem 15 ).

\section{Chaotic expansion for the local time of integrators  }
\begin{flushleft}

Consider the Gaussian integrator $x$  given by \eqref{d} where the wite  noise $\xi$ is generated by the Wiener process $ \lbrace w(t), t \in [0,1] \rbrace$ . Assume that  $A$ is continuously invertible and $ \| A \| \leqslant 1 $. Denote by $\ell(u)$ the local time of the integrator $x$ in $u$ up to time $1$. As a white noise functional, this local time has a chaotic expansion which is our subject of investigation in this section. The following lemma gives the answer. 

\end{flushleft}

We use the notation $\sigma(t) $ instead of $\| A\mathbf{1}_{[0,t]} \|.$ We also put $  t_{n}^{*} = \max\lbrace t_{1},\ldots,t_{n} \rbrace  $ and recall that $ \displaystyle  H_{n} $ is the n-th Hermite polynomial, already defined in section 1 . 

\begin{lem}
 For every $ u \in \R $, the It\^{o} Wiener expansion of  $\displaystyle \ell(u)$ is 
 \[ \ell(u)= \sum_{n=0}^{\infty} \Gamma (A)\left( \int _{\Delta_{n}}a_{n}(t_{1},\ldots,t_{n})dw(t_{1})\ldots dw(t_{n})\right) \]
 where the symmetric kernels $\lbrace a_n \rbrace  $ are given by 
 \[ a_{n}(t_{1},\ldots,t_{n})= \int_{t_{n}^{*}}^{1}\dfrac{1}{\sigma(t)^{n}}H_{n}\left( \dfrac{u}{\sigma(t)}\right)p_{\sigma(t)^{2}}(u)\mathrm{d}t. \]

\end{lem}
\begin{proof}

The local time of $x$ at the point $u$ is given by 
\[ \ell(u)= L^{2}-\lim _{\varepsilon \rightarrow 0^{+}} \int_{0}^{1}p_{\varepsilon}(x(t)-u)\mathrm{d}t . \]
The integrand in the right hand side has the chaotic expansion 
\[p_{\varepsilon}(x(t)-u)=  \sum_{n=0}^{\infty} \dfrac{1}{n!} \left( \dfrac{\sigma(t)}{\sqrt{\sigma(t)^{2}+\varepsilon} }\right)^{n} H_{n}\left( \dfrac{u}{\sqrt{\sigma(t)^{2}+\varepsilon}}\right)p_{\sigma(t)^{2}+\varepsilon}(u) H_{n} \left( \dfrac{x(t)}{\sigma(t)}\right)  \]
In fact,
$$
\int_{\R}p_{\varepsilon}(\sigma(t)v-u)H_{n}(v)p_{1}(v)dv = (-1)^{n}\int_{\R}p_{\varepsilon}(\sigma(t)v-u)p_{1}^{(n)}(v)dv = $$
$$  = \int_{\R}\dfrac{\partial ^{n}}{\partial v^{n}}p_{\varepsilon}(\sigma(t)v-u)p_{1}(v)dv  = (-\sigma(t)) ^{n} \int_{\R}\dfrac{\partial ^{n}}{\partial u^{n}} p_{\varepsilon}(\sigma(t)v-u)p_{1}(v)dv =$$
$$ = (-\sigma(t)) ^{n} \dfrac{\partial ^{n}}{\partial u^{n}} \int_{\R}p_{\varepsilon}(\sigma(t)v-u)p_{1}(v)dv = (-\sigma(t)) ^{n} \dfrac{\partial ^{n}}{\partial u^{n}} p_{\varepsilon + \sigma(t)^{2}}(u) = $$ 
$$ =\left( \dfrac{\sigma (t)}{\sqrt{\sigma(t)^{2}+\varepsilon}}\right) ^{n} H_{n}\left( \dfrac{u}{\sqrt{\sigma(t)^{2}+\varepsilon}} \right)p_{\sigma(t)^{2}+\varepsilon}(u). $$
Now we obtain 
\begin{equation}
\label{g}
 \ell(u)=\sum _{n=0}^{\infty}\dfrac{1}{n!}\int_{0}^{1}H_{n}\left( \dfrac{u}{\sigma(t)} \right)p_{\sigma(t)^{2}}(u) H_{n}\left( \dfrac{x(t)}{\sigma(t)} \right)\mathrm{d}t 
\end{equation}
This series has orthogonal summands and converges in $ L_{2}$. Orthogonality can be easily checked using the formula \cite{10}  
$$ \displaystyle \mathbb{E} \left[ H_{n}\left( \dfrac{x(t)}{\sigma(t)} \right) H_{m}\left( \dfrac{x(s)}{\sigma(s)} \right)\right] = 0 \quad \forall \quad n \neq m . $$
Concerning the convergence, we use the following formula \cite{10}  

\[ \mathbb{E} \left[ H_{n}\left( \dfrac{x(t)}{\sigma(t)} \right) H_{n}\left( \dfrac{x(s)}{\sigma(s)} \right)\right] = n! \lambda ^{n} \]
where $$ \displaystyle \lambda = \dfrac{\left(  A\mathbf{1}_{[0,t]},A\mathbf{1}_{[0,s]} \right)}{\sigma(t)\sigma(s)} , $$ 
and the estimate  \cite{8} 
\begin{equation}
\label{h}
 \forall\; \alpha \; \in\; \left[ \dfrac{1}{4},\dfrac{1}{2}\right] \quad \exists \; c\;>\; 0 \; ; \; \sup_{x \in \R} \left | H_{n}(x)e^{-\alpha x^{2}} \right |\; \leq \; c \; \sqrt{n!} \; n^{-\frac{8\alpha -1}{12}}\quad \forall \; n \; \in \mathbb{N}^{*}   
\end{equation} 
\begin{flushleft}
 to obtain
\end{flushleft}
\[ \dfrac{1}{n!^{2}} \mathbb{E}\left[ \int_{0}^{1}H_{n}\left( \dfrac{u}{\sigma(t)} \right)p_{\sigma(t)^{2}}(u) H_{n}\left( \dfrac{x(t)}{\sigma(t)} \right)\mathrm{d}t\right] ^{2} \leqslant  \]
\[ \leqslant \dfrac{C}{\sqrt{n}} \int_{0}^{1}\int_{0}^{1} \dfrac{|\lambda^{n}| }{\sigma(s)\sigma(t)} \mathrm{d}t \mathrm{d}s = \dfrac{2C}{\sqrt{n}} \int_{0}^{1}\int_{s}^{1} \dfrac{| \lambda |^{n} }{\sigma(s)\sigma(t)} \mathrm{d}t \mathrm{d}s\]
for some constant $ C > 0 $.\\
Let us introduce the function 
 
\[f (x)= \sum _{n=1}^{\infty}\dfrac{x^{n}}{\sqrt{n}} \quad , \; 0\leqslant x < 1  .\]
We have 
$$ \dfrac{f(x)}{1-x}=\sum _{n=1}^{\infty}\left( \sum_{k=1}^{n}\dfrac{1}{\sqrt{k}} \right) x^{n} . $$
By using the inequality 
 $$  2(\sqrt{n+1}-1) = \int_1^{n+1}\dfrac{dx}{\sqrt{x}} \; \leqslant \; \sum_{k=1}^{n}\dfrac{1}{\sqrt{k}} $$
we can get 
$$ \dfrac{f'(x)}{f(x)+2}\; \leqslant  \; \dfrac{1}{2(1-x)} .$$
Which leads to  
$$ f(u) \; \leqslant \; \dfrac{2}{\sqrt{1-u}} - 2 \qquad \forall \quad 0 \; \leqslant \; u \; < \; 1  .$$
The integrability of $\displaystyle \int_{0}^{1}\int_{s}^{1} \dfrac{1 }{\sigma(s)\sigma(t)} \mathrm{d}t \mathrm{d}s $ is a consequence of the inequality 
\begin{equation}
\label{11}
\sigma(t)\geqslant \dfrac{\sqrt{t}}{\parallel A^{-1} \parallel} .
\end{equation}
Now it suffices to prove the convergence of the integral 
$$ \int_{0}^{1}\int_{s}^{1} \dfrac{1}{\sqrt{1-|\lambda |}\; \sigma(s)\sigma(t)} \; \mathrm{d}t \mathrm{d}s  $$
Let us use the inequality 
$$ G \left(  A\mathbf{1}_{[0,s]},A\mathbf{1}_{[0,t]} \right) \; \geqslant \; c(2)G \left( \mathbf{1}_{[0,s]},\mathbf{1}_{[0,t]} \right) $$
from \cite{7}, where $G$ denotes the Gram determinant. One can check that for $s<t$
$$ \dfrac{1}{\sqrt{1-|\lambda |}\; \sigma(s)\sigma(t)} = \dfrac{\sqrt{ 1+| \lambda| }}{\sqrt{G \left(  A\mathbf{1}_{[0,s]},A\mathbf{1}_{[0,t]} \right)}} $$
 $$ \leqslant \dfrac{\sqrt{2}}{\sqrt{c(2)}\; \sqrt{s(t-s)}} $$
and the latter expression is clearly integrable.

Moreover, the orthogonal expansion given by \eqref{g}  is simply the It\^{o} Wiener expansion we are looking for. Indeed, the n-th term of the sum is the action on the white noise of the following symmetric Hilbert Shmidt form 
$$  A_{n}(h_{1},\cdots,h_{n}) = \dfrac{1}{n!}\int_{0}^{1}H_{n}\left( \dfrac{u}{\sigma(t)} \right)p_{\sigma(t)^{2}}(u) \dfrac{1}{\sigma(t)^{n}}(A\mathbf{1}_{[0,t]})^{\bigotimes n}(h_{1},\cdots,h_{n})\mathrm{d}t = $$
$$ = \dfrac{1}{n!}\int_{0}^{1}H_{n}\left( \dfrac{u}{\sigma(t)} \right)p_{\sigma(t)^{2}}(u) \dfrac{1}{\sigma(t)^{n}}(\mathbf{1}_{[0,t]})^{\bigotimes n}(A^{*}h_{1},\cdots,A^{*}h_{n})\mathrm{d}t = $$
$$ = \dfrac{1}{n!}\int_{[0,1]^{n+1}} H_{n}\left( \dfrac{u}{\sigma(t)} \right)p_{\sigma(t)^{2}}(u) \dfrac{1}{\sigma(t)^{n}}\mathbf{1}_{[0,t]^{n}}(\vec{s})(A^{*}h_{1})(s_{1})\cdots (A^{*}h_{n})(s_{n})\mathrm{d}t \mathrm{d}\vec{s} = $$
$$ = \dfrac{1}{n!}\int_{[0,1]^{n}} \left[ \int_{s_{n}^{*}}^{1} H_{n}\left( \dfrac{u}{\sigma(t)} \right)p_{\sigma(t)^{2}}(u) \dfrac{1}{\sigma(t)^{n}} \mathrm{d}t \right](A^{*}h_{1})(s_{1}) \cdots   ( A^{*}h_{n})(s_{n}) \mathrm{d}\vec{s} = $$
$$  = B_n^A (h_1,\ldots,h_n)$$
 where $ s_{n}^{*}=\max \lbrace s_{1},\ldots,s_{n}\rbrace ,  \vec{s}=(s_{1},\ldots,s_{n}), \; h_{1},\cdots,h_{n} \in L_{2}([0,1]) $ and $ B_n $ is the symmetric Hilbert Shmidt form associated to the kernel 
$$ b_{n}(\vec{s}) = \dfrac{1}{n!} \int_{s_{n}^{*}}^{1} H_{n}\left( \dfrac{u}{\sigma(t)} \right)p_{\sigma(t)^{2}}(u) \dfrac{1}{\sigma(t)^{n}} \mathrm{d}t = \dfrac{1}{n!}a_{n}(\vec{s}) \; .$$
Notice that the square integrability of $ b_n $ can be confirmed using the estimate \eqref{h} and the inequality \eqref{11}. Really,
$$\int_{[0,1]^{n}} b_n^2(\vec{s})\mathrm{d}\vec{s} = \dfrac{1}{n!} \int_{\Delta _n (1)}\left[  \int_{s_{n}^{*}}^{1} H_{n}\left( \dfrac{u}{\sigma(t)} \right)p_{\sigma(t)^{2}}(u) \dfrac{1}{\sigma(t)^{n}} \mathrm{d}t \right] ^{2} \mathrm{d}\vec{s} = $$
$$  = \dfrac{n}{2 \pi \; n!^{\; 2}} \int_{0}^1 s^{n-1 }  \left[ \int_{s}^{1}  H_{n} \left( \dfrac{u}{\sigma(t)} \right) \exp \left( -\dfrac{1}{2} \left( \dfrac{u}{\sigma(t)} \right)^2 \right) \dfrac{1}{ \sigma(t)^{n+1} } \; \mathrm{d}t \right]  ^2 \; \mathrm{d}s \leqslant $$
$$  \leqslant \dfrac{c \; n }{2\pi \; n! \; \sqrt{n}} \int_{0}^1 s^{n-1 } \left[ \int_{s}^{1} \dfrac{1}{ \sigma(t)^{n+1} } \; \mathrm{d}t \right]  ^2 \; \mathrm{d}s \leqslant $$
$$ \leqslant \dfrac{c \; \sqrt{ n} \;\parallel A^{-1} \parallel ^{2n+2} }{2\pi\; n!} \int_{0}^1 s^{n-1 } \left[   \int_{s}^{1} \dfrac{1}{ \sqrt{t}^{\;n+1} } \; \mathrm{d}t \right] ^{2}\; \mathrm{d}s \leqslant $$
\begin{equation}
\label{i}
  \leqslant  c' \dfrac{ \sqrt{ n} \;\parallel A^{-1} \parallel ^{2n+2} }{n! \;  (n-1)^{2}}  \left( 1+ \dfrac{1}{n} - \dfrac{4}{n+1} \right) .
\end{equation}

It follows that 
$$ \dfrac{1}{n!}\int_{0}^{1}H_{n}\left( \dfrac{u}{\sigma(t)} \right)p_{\sigma(t)^{2}}(u) H_{n}\left( \dfrac{x(t)}{\sigma(t)} \right)\mathrm{d}t = A_{n}( \xi,\ldots,\xi) = $$
$$ =  \Gamma (A)\left( \int _{\Delta_{n}}a_{n}(t_{1},\ldots,t_{n})dw(t_{1})\ldots dw(t_{n})\right) $$
which ends the proof.    
\end{proof}

\section{Integral representations for the local time of integrators}

 We present here some integral representations for the local time of integrators. A first representation is devoted to the Wiener process. Denote by $\ell_{w}(u,t)$ the local time of the Wiener process $(w(t))_{t\geq 0}$ at point $u $ up to time $t$. As we mentioned in the introduction, it is an $L^2-$limit
$$
\ell_w(u,t)=\lim_{\varepsilon\to 0}\ell_{w,\varepsilon}(u,t),
$$
where $ \displaystyle \ell_{w,\varepsilon}(u,t)=\int^t_0 p_\varepsilon (w(r)-u)dr.$ 

\begin{lem}
The Clark representation formula for the local time of the Wiener process is
$$
\ell_w(u,t)=\int^t_0 p_r(u)dr + \int^t_0 \int^t_r p'_{s-r}(w(r)-u)ds dw(r)
$$
\end{lem}

\begin{proof}
The Clark theorem asserts that there exist square integrable $w-$adapted processes $(u(r))_{r\in [0,t]},$ $(u_{\varepsilon}(r))_{r\in [0,t]},$ such that
$$
\ell_w(u,t)=\int^t_0 p_r(u)dr + \int^t_0 u(r)dw(r);
$$
$$
\ell_{w,\varepsilon}(u,t)=\int^t_0 p_{r+\varepsilon}(u)dr + \int^t_0 u_{\varepsilon}(r)dw(r).
$$
The $L^2-$convergence  $\ell_{w,\varepsilon}(u,t)\to \ell_{w}(u,t),$ $\varepsilon\to 0,$  implies that 
$$
\mathbb{E}\int^t_0 (u_{\varepsilon}(r)-u(r))^2 dr\to 0, \varepsilon\to 0.
$$
Let us find processes $u_{\varepsilon}.$ As random variables $\ell_{w,\varepsilon}(u,t)$ are stochastically differentiable, the Clark-Ocone formula \cite[Prop.1.3.14]{10} is applicable
$$
u_{\varepsilon}(r)=\mathbb{E}[D(\ell_{w,\varepsilon}(u,t))(r)/\mathcal{F}^w_r],
$$
where $\mathcal{F}^w_r$ is the $\sigma-$filed generated by $w(s),s\leq r.$ Now, the stochastic derivative of $\ell_{w,\varepsilon}(u,t)$ equals
$$
D(\ell_{w,\varepsilon}(u,t))(r)=\int^t_r p'_\varepsilon (w(s)-u)ds.
$$
To calculate its conditional expectation we will use the independence of increments of $w $
$$
\mathbb{E} \left[ D\ell_{w,\varepsilon}(u,t)(r) / \mathcal{F}^w_r \right] = \int^t_r \int_{\mathbb{R}} p'_{\varepsilon }(w(r)-u+y)p_{s-r}(y)dy ds =
$$
$$
=-\int^t_r\int_{\mathbb{R}}  p_\varepsilon (w(r)-u+y)p'_{s-r}(y)dyds.
$$
Consequently, 
$$
u_{\varepsilon}(r)=-\int^t_r\int_{\mathbb{R}}  p_\varepsilon (w(r)-u+y)p'_{s-r}(y)dyds.
$$
Taking the limit $\varepsilon\to 0$ one obtains the needed expression for $u.$
\end{proof}

The following integral representation does not concern only the Wiener process but a large class of integrators. For this, let us consider the integrator $x$ with the representation \eqref{d} where $A$ is a continuously invertible, $\| A\| \leqslant 1 $  and $\xi$ is generated by the Brownian motion $ \displaystyle \lbrace w(t), \; 0\leqslant t \leqslant 1  \rbrace .$  It was established in \cite{7} that $x$ has a local time $\ell(u,t)$ (at the point $u$ up to time $t$ ) which can be defined as a density of the occupation measure, has a continuous modification and satisfies the condition
 $$
 \mathbb{E}\int_{\mbR}\ell(u, 1)^2du<+\infty.
 $$
 Define $\sigma^2(t)=\|A \mathbf{1}_{[0; t]}\|^2=\mathbb{E} x(t)^2, t\in[0; 1].$ Suppose that $\sigma^2$ is nondecreasing and 
the following integral converges
\begin{equation}
\label{20}
\int^t_0 \int^t_r \int^t_r \frac{r^3}{(\sigma^2(u)\sigma^2(v)-\sigma^4(r))^{\frac{3}{2}}}dudvdr<\infty.
\end{equation} 
 Then the following statement holds.
 \begin{thm}
The next equality is true
$$
\ell(u, t)=\int^t_0p_{\sigma^2(s)}(u)ds+\int^t_0\int^t_rp'_{\sigma^2(s)-\sigma^2(r)}(x(r)-u)dsdx(r).
$$
Here the symbol ${}'$ is used for the derivative with respect to a spatial variable and the integral against $dx(t)$ is the extended stochastic integral.
\end{thm}
 \begin{proof}
 As it was proved in \cite{7} that the local time $\ell(u, t)$ can be obtained as a density of an occupation measure. Consequently the Fourier transform of $\ell(\cdot, t)$ looks like
  $$
  \int_{\mbR}e^{i\lambda u}\ell(u, t)du=\int^t_0e^{i\lambda x(s)}ds.
  $$
  Note that
  $$
  e^{i\lambda x(s)+\frac{\lambda^2}{2}\sigma^2(s)}=\Gamma(A^*)e^{i\lambda w(s)+\frac{\lambda^2s}{2}}.
  $$
  The last expression can be transformed into
  $$
  \Gamma(A^*)e^{i\lambda w(s)+\frac{\lambda^2s}{2}}=
  \Gamma(A^*)\left( 1+i\lambda\int^s_0e^{i\lambda w(r)+\frac{\lambda^2r}{2}}dw(r)\right) =
  $$
  $$
  =1+i\lambda\int^s_0e^{i\lambda x(r)+\frac{\lambda^2}{2}\sigma^2(r)}dx(r).
  $$
  Consequently,
  $$
  e^{i\lambda x(s)}=e^{-\frac{\lambda^2\sigma^2(s)}{2}}+i\lambda\int^s_0e^{i\lambda x(r)}e^{-\frac{\lambda^2}{2}(\sigma^2(s)-\sigma^2(r))}dx(r).
  $$
  The last integral is the extended stochastic integral, which was defined above. Here the Fourier transform of the local time $\ell(\cdot, t)$ has a form
  $$
  \int_{\mbR}e^{i\lambda u}\ell(u, t)du=\int^t_0e^{-\frac{\lambda^2\sigma^2(s)}{2}}ds+
  $$
  $$
  +i\lambda\int^t_0\int^s_0e^{i\lambda x(r)}e^{=\frac{\lambda^2}{2}(\sigma^2(s)-\sigma^2(r))}dx(r)ds.
  $$
  In the last integral one can change the order of integration due to the properties of the extended stochastic integral (it is a closed linear operator \cite{3, 4} and one can approximate the external integral by the Riemanian sums). Finally
  \begin{equation}
  \label{f}
  \int_{\mbR}e^{i\lambda u}\ell(u, t)du=\int^t_0e^{-\frac{\lambda^2\sigma^2(s)}{2}}ds+
  \end{equation}
  $$
  +
  i\lambda\int^t_0\int^t_r e^{i\lambda x(r)}e^{-\frac{\lambda^2}{2}(\sigma^2(s)-\sigma^2(r))}dsdx(r).
  $$
   Now one can apply the inverse Fourier transform to both sides of \eqref{f} and inside of the integral (using the same arguments as above) and obtain the desired equality.
\end{proof}

Notice that the condition \eqref{20} is sufficient to give a meaning to the extended stochastic integral in the theorem. But as it seems to be a little bit difficult to check, we give, in the following lemma, an easy sufficient condition which implies \eqref{20}.

\begin{lem}
 Assume that 
\begin{enumerate}
\item the function $\sigma^2$ is absolutely continuous;

\item its derivative satisfies
$$
\frac{d\sigma^2(t)}{dt}\geq f(t)
$$
for some nonegative increasing function $f$

\item 
$$
\int^t_0 \bigg(\frac{r}{f(r)}\bigg)^{\frac{3}{2}}dr<\infty.
$$

\end{enumerate}

Then the condition \eqref{20} holds.
\end{lem}

\begin{proof}
 The denominator in \eqref{20} can be estimated in the following way
$$
\sigma^2(u)\sigma^2(v)-\sigma^4(r)\geq \sigma^2(r)(\sigma^2(u)+\sigma^2(v)-2\sigma^2(r))\geq
$$
$$
\geq \mbox{const} \ r f(r) (u+v-2r).
$$
Now it is clear that condition 3 of the lemma implies \eqref{20}.
\end{proof}

It is important to state that Theorem 15 can also be derived from the chaotic expansion in Lemma 13, under the condition \eqref{20}.

Finally, as we stated in the first section the existence of the minimal norm integral representation, we will now apply it for the local time 
$$
\ell(u,t)=\int^t_0 \delta_u(x(r))dr
$$
of the integrator $(x(t))_{t\in [0,1]}.$ Recall that $x(t)=(A\mathbf{1}_{[0,t]},\xi),$ where $\xi$ is a white noise in $L_2([0,1])$ and $A$ is a continuous and continuously invertible linear operator in $L_2([0,1]).$ The local time $\ell(u,t)$ is the following $L^2-$limit
$$
\ell(u,t)=L^2-\lim _{\varepsilon \to 0}\ell_{\varepsilon}(u,t),
$$
where $\displaystyle \ell_{\varepsilon}(u,t)=\int^t_0 p_\varepsilon(x(r)-u)dr.$ As before, we will use the notation $\sigma^2(t)=\| A\mathbf{1}_{[0,t]} \|^2.$ 

\begin{thm}
\label{min_norm_representation}
For the local time $\ell(u,t)$ of the integrator $x$ one has
$$
d(\ell(u,t))=A\bigg(1_{\cdot<t}\int^t_\cdot\frac{1}{\sigma^2(r)} e^{\frac{x(r)^2-u^2}{2\sigma^2(r)}}\bigg(\Phi\bigg(\frac{x(r)}{\sigma(r)}\bigg)-1_{x(r)>u}\bigg) dr \bigg).
$$

Respectively, the minimal norm integral representation of $\ell(u,t)$ is
$$
\ell(u,t)=\int^t_0 p_{\sigma^2(s)}(u)ds + \int^t_0\int^t_s\frac{1}{\sigma^2(r)} e^{\frac{x(r)^2-u^2}{2\sigma^2(r)}}\bigg(\Phi\bigg(\frac{x(r)}{\sigma(r)}\bigg)-1_{x(r)>u}\bigg) dr dx(s).
$$

\end{thm}

\begin{proof}

The operator $d$ is continuous, so 
$$
d(\ell(u,t))=\lim_{\varepsilon\to 0} d(\ell_{\varepsilon}(u,t))=\lim_{\varepsilon\to 0}\int^t_0 d(p_\varepsilon(x(r)-u))dr.
$$According to the lemma \ref{lem5}
$$
d(p_\varepsilon(x(r)-u))=d(p_\varepsilon((A1_{[0,r]},\xi)-u))=
$$
$$
=\bigg(\int_{\mathbb{R}}p_\varepsilon(y-u)\frac{1}{\sigma^2(r)}e^{\frac{x(r)^2-y^2}{2\sigma^2(r)}}\bigg(\Phi\bigg(\frac{x(r)}{\sigma(r)}\bigg)-1_{x(r)>y}\bigg)dy\bigg)A1_{[0,r]}.
$$
Consequently,
$$
d(\ell_{\varepsilon}(u,t))=\int^t_0 \bigg(\int_{\mathbb{R}}p_\varepsilon(y-u)\frac{1}{\sigma^2(r)}e^{\frac{x(r)^2-y^2}{2\sigma^2(r)}}\bigg(\Phi\bigg(\frac{x(r)}{\sigma(r)}\bigg)-1_{x(r)>y}\bigg)dy\bigg)A1_{[0,r]} dr=
$$
$$
=A\bigg(1_{\cdot<t}\int^t_\cdot\frac{1}{\sigma^2(r)}  \int_{\mathbb{R}}p_\varepsilon(y-u)e^{\frac{x(r)^2-y^2}{2\sigma^2(r)}}\bigg(\Phi\bigg(\frac{x(r)}{\sigma(r)}\bigg)-1_{x(r)>y}\bigg)dy dr \bigg).
$$
Taking the limit $\varepsilon\to 0$ one obtains the needed expression
$$
d(\ell(u,t))=A\bigg(1_{\cdot<t}\int^t_\cdot\frac{1}{\sigma^2(r)} e^{\frac{x(r)^2-u^2}{2\sigma^2(r)}}\bigg(\Phi\bigg(\frac{x(r)}{\sigma(r)}\bigg)-1_{x(r)>u}\bigg) dr \bigg).
$$
The theorem is proved.

\end{proof}

\end{document}